\documentclass{amsart}
\usepackage{amsfonts}
\usepackage{amsmath}
\usepackage{amssymb}
\usepackage{amsthm}
\usepackage[pagewise,mathlines]{lineno}
\usepackage{fullpage}
\usepackage{graphicx}
\usepackage{amsfonts}
\usepackage[stable]{footmisc}

\numberwithin{equation}{section}

\newcommand{\N}{\mathbb{N}}

\newtheorem{theorem}{Theorem}[section]
\newtheorem{proposition}[theorem]{Proposition}
\newtheorem{lemma}[theorem]{Lemma}
\newtheorem{corollary}[theorem]{Corollary}

\newtheorem{remark}[theorem]{Remark}

\theoremstyle{definition}

\newtheorem{rem}{Remark}[section]

\begin{document}

%
%

\title[A Variant of Harborth Constant]{A Variant of Harborth Constant}

\author{A.  Lemos, B.K. Moriya, A.O. Moura and A.T. Silva$^{\ast}$}
\thanks{$\ast$ The authors were partially supported by FAPEMIG grant APQ-02546-21 and FAPEMIG grant RED-00133-21.}
\address{Departamento de Matem\'{a}tica, Universidade Federal de Vi\c cosa, Vi\c cosa-MG, Brazil}

\email{abiliolemos@ufv.br\\bhavinkumar@ufv.br\\allan.moura@ufv.br}
\email{anderson.tiago@ufv.br}

\keywords{Finite abelian group, subsets, $k$-Harborth constant}

\subjclass[2010]{20K01,11B75}

\begin{abstract}
Let $G$ be a finite additive abelian group. For given $k$ a positive integer, the $k$-Harborth constant $g^k(G)$ is defined to be the smallest positive integer $t$ such that given a set $S$ of elements of $G$ with size $t$ there exists a zero-sum subset of size $k$. We find either the exact value of $g^k(G)$, or  lower and upper bounds for this constant for some groups. 
\end{abstract}

\maketitle

\section{Introduction}
\hspace{0.6cm}Let $G$ be a finite additive abelian group with exponent
$n$ and $S$ be a sequence over $G$. The enumeration of subsequences
with certain prescribed properties is a classical topic in Combinatorial
Number Theory going back to Erd\H{o}s, Ginzburg and Ziv (see \cite{EGZ,Ger1,Ger2})
who proved that $2n-1$ is the smallest integer, such that every sequence
$S$ over a cyclic group $C_{n}$ has a subsequence of length $n$
with zero-sum. This raises the problem of determining the smallest
positive integer $l$, such that every set $S=\{g_{1},\cdots, g_{l}\}$
has a nonempty zero-sum subset of size $\exp(G).$ Such an integer $l$ is called
the {\it Harborth constant of $G$} (see \cite{harborth} ), denoted by $g(G)$, since Harborth was the first to introduce this constant for the group $C_n^d$ and particular results were obtained by him. Some important results on this constant were obtained in the last 50 years. Kemnitz, see \cite{kemnitz}, proved the following trivial bound,

\[(n-1)2^{d-1}+1\le g(C_n^d)\le (n-1)n^{d-1}+1, n>2.\]

For $d=1,$ the upper bound above does not holds when $n$ is an even number. In fact, when $n$ is an even number, he proved that $g(C_n^d)\ge n\cdot 2^{d-1}+1$. In the same paper, Kemnitz also studied this constant for $d=2$ and computed the exact values for $p=3,5,7.$ Indeed, he proved that $g(C_p^2)= 2p-1,$ for $p\in\{3,5,7\}.$  
Several results were obtained over the years by several authors, for example $g(C_3^3)=10$, $g(C_3^4)=21$ and $g(C_3^5)=45$ (see \cite{bre,bb,fgr,harborth,kemnitz,kem2}). Gao and Thangadurai (2004), see \cite{gaoth}, proved that $g(C_p^2)= 2p-1$ for all primes numbers $p\geq67$ (this was later refined to $p\ge47,$ see \cite{gaogs}). They also proved that $g(C_n^2)= 2n+1$ for particular values of $n,$ where $n$ is an even number. Hence they propose the following conjecture:
\[g(C_n^2)=\begin{cases}2n-1,\mbox{ if }n\mbox{ is odd}\\
2n+1,\mbox{ if }n \mbox{ is even.}
\end{cases}\]
Actually, it is known (see Lemma 10.1 in \cite{gaog}) that $g(G)=|G|+1$ if and only if $G$ is an elementary $2$-group or a cyclic group of even order. Of course, when $G=C_2^r$ there is no zero-sum subset of size $2,$ so by vacuity $g(G)=|G|+1.$ Meshulam, see \cite{mes}, proved that $g(C_3^d)\le (1+o(1))\dfrac{3^d}{d}$. In 2013, Marchan {\it et al.}, see \cite{marchan}, proved that,
 \[g(C_2\oplus C_{2n})=\begin{cases}2n+3,\mbox{ if }n\mbox{ is odd}\\
2n+2,\mbox{ if }n \mbox{ is even.}
\end{cases}\]
 In 2019, Guillot \textit{et al.}, see \cite{guil}, proved that,
\[g(C_3\oplus C_{3p})=\begin{cases}3p+3,\mbox{ if }p\neq 3\\
13,\mbox{ if }p = 3,
\end{cases}\]
where $p$ is a prime number.

This motivates us to study the variant of Harborth constant which we will introduce next.
For given $k\in \N,$ we define $g^k(G)$ to be the smallest positive integer $t$ such that given a set $S$ of elements of $G$ with size $t$ there exists a zero-sum subset of size $k$. Here, we shall be studying this constant for all $k\in\mathbb{N}, k\ge3$ and $G=C_p,$ where $p$ is a prime number. We will also be proving some lower bounds and equalities for any $k\ge 1$ and any group $G.$ 

\section{Notations and terminologies}

\global\long\def\labelenumi{(\roman{enumi})}

\hspace{0.6cm} In this section, we will introduce some notations and terminologies. Let $\mathbb{N}_{0}$ be {\it the set
of nonnegative integers}. For integers $a,b\in\mathbb{N}_{0}$, we
define $[a,b]=\left\{ x\in\mathbb{N}_{0}:a\leq x\leq b\right\} $. For a set $S=\{g_1,\dots,g_m\} \subset G,$ we define 
\begin{enumerate}
\item $\left|S\right|=m$ {\it the size of $S$};
\item a {\it $k$-subsum} of $S$ is a sum of the form $\sigma\left(T\right)=\sum_{i\in I_{T}}g_{i}$, where $T$ is a subset of $S$ and $\emptyset\neq I_{T}\subset [1,m],$ with $|I_{T}|=k;$
\item $\sum_{\le k}\left(S\right)=\left\{ \sum_{i\in I}g_{i}:\emptyset\neq I\subseteq\left[1,m\right],\,|I|\le k\right\} $ to be the {\it set of nonempty subsums of maximum size $k$} of $S;$
\item $\sum_{ k}\left(S\right)=\left\{ \sum_{i\in I}g_{i}:\emptyset\neq I\subseteq\left[1,m\right],\,|I|=k\right\} $ to be the {\it set of nonempty $k$-subsums} of $S.$ 
\end{enumerate}

The set $S$ is called 
\begin{enumerate}
\item a {\it $m$-zero-sum set} if $\sigma\left(S\right)=0;$
\item a {\it $k$-zero-sum free set} if $0\notin\sum_{ k}\left(S\right).$ 
\end{enumerate}
\vspace{0cm}

\section{The $k$-Harborth Constant }\hspace{0.6cm}

In this section we will study a variant of Harborth. By $g^{k}(G)$ we denote the smallest positive integer $\ell$ such that any subset $S$ of $G,$ with $|S|\geq \ell,$ contains a $k$-zero-sum subset $T.$ We call this constant by $k$-{\it Harborth Constant}. Recall that $g^{k}(G)=g(G),$ when $k=\mbox{exp}(G),$ where $g(G)$ is an usual Harborth constant. More precisely, we shall be studying the underlined constant for all $k\in\mathbb{N}, k\ge3$ and $G=C_p,$ where $p$ is a prime number. We will also be proving some lower bounds and equalities for any $k\ge 1$ and any group $G.$ Clearly, $g^1(G)=|G|$. Hence from here on we shall be assuming $k\ge 2$. 

\subsection{General results} \hspace{0.6cm}

We start with a simple lemma.
\begin{lemma}
Given a finite abelian group $G,$ we have $g^2(G)=|A|+|B|+1,$ where $B=\{x\in G\ | \ 2x=0\}$ and $G\setminus B = A\cup -A.$

\end{lemma}
\begin{proof}
We consider the set $S=B\cup A$ and note that $S$ is a $2$-zero-sum free set, because all inverses belong $-A.$ Now, any set of size $|A|+|B|+1$ has at least two elements, say $x,y,$ such that $x+y=0.$ 
\end{proof}

\begin{remark}
If $|G|$ is odd, then $g^2(G)=\dfrac{|G|+3}{2}. $
Indeed one can see that the cardinalities of the sets defined in the lemma above are given by $|A|=\dfrac{|G|-1}{2}$ and $|B|=1.$
\end{remark}

\begin{proposition} \label{lower}
Let $G=C_{n_1}\oplus C_{n_2}\oplus\ldots \oplus C_{n_r}$, with $n_i|n_{i+1}$ and $n_1\geq 4$. Then 

\[g^3(G) \geq \left\{\begin{array}{ll}
1+\sum_1^t \left\lceil \dfrac{n_i}{3}\right\rceil+\sum_{t+1}^r \dfrac{n_i}{2} & \mbox{if }  2\not|n_t \mbox{ and } 2|n_{t+1},\\
& \\

2+\sum_1^r \left\lceil \dfrac{n_i}{3}\right\rceil &  \mbox{if }  n_r \mbox{ is odd.}\\
\end{array} \right.\]

\end{proposition}

\begin{proof}
First we observe two statements:

I) For $n\geq 4,$ we consider 
\[
S=\left\{1,\dots,n-\left\lfloor \dfrac{2n}{3}\right\rfloor\right\}.
\]
Since \[n-\left\lfloor \dfrac{2n}{3}\right\rfloor+n-\left\lfloor \dfrac{2n}{3}\right\rfloor-1+n-\left\lfloor \dfrac{2n}{3}\right\rfloor-2<n, \mbox{ for any }n\ge 4.\]

II) For each $C_{n}$ with $n$ even, consider a set $S=\{1,3,5,\ldots,n-1\}$. Clearly, $S$ is a $3$-zero-sum free set.

Now, if $n_i$ is odd set $k_i=n_i-\left\lfloor \dfrac{2n_i}{3}\right\rfloor = \left\lceil \dfrac{n_i}{3}\right\rceil$, for each $i\ge 1$ and apply statement I) to get a subset $S_i=\prod_{j=1}^{k_i} je_i,$ where $S_i$ is a $3$-zero-sum free set.

If $n_i$ is even set $k_i=\dfrac{n_i}{2}$ and apply statement II) to get  a subset $S_i=\prod_{j=1}^{k_i} {(2j-1)}e_i,$ where $S$ is a $3$-zero-sum free set.

For the first inequality, the set $S=\prod_{1}^{r} S_i$ is a $3$-zero-sum free set.

For the second inequality, consider the set $S=\prod_{1}^{r} S_i.$ Then $0\cdot S$ is a $3$-zero-sum free set.

\end{proof}
%


\begin{proposition}\label{p5}
Let $G=C_{n_1}\oplus C_{n_2}\oplus \ldots\oplus  C_{n_t},\mbox{ where } n_i|n_{i+1}$. Then
$g^k(G) \geq \dfrac{|G|}{n_t}\cdot\left(\left\lfloor \dfrac{n_t}{k} \right\rfloor-\epsilon\right)+1$, where $k\in \mathbb{N}, k<n_t,$ and
$\epsilon=\begin{cases}
1, \mbox{ if }k|n_t\\
0, \mbox{ otherwise}
\end{cases}.$\\
\end{proposition}

\begin{proof}
Consider a set \[S=\left\{(g,h)|g\in C_{n_1}\oplus C_{n_2}\oplus \ldots\oplus  C_{n_{t-1}}\mbox{ and }h\in \left[1,\left\lfloor \dfrac{n_t}{k}\right \rfloor-\epsilon\right]\right\},\]
where $\epsilon=\begin{cases}
1, \mbox{ if }k|n_t\\
0, \mbox{ otherwise}
\end{cases}.$\\
Notice that $|S|=\dfrac{|G|}{n_t}\cdot\left(\left\lfloor \dfrac{n_t}{k} \right\rfloor-\epsilon\right)$ and $S$ is a $k$-zero-sum free set. Hence we are done.
\end{proof}

\begin{remark}
Note that Proposition \ref{p5} gives far better lower bound compare to Proposition \ref{lower}, where $r\ge 2$. More precisely, for $n\ge 3r/2,\mbox{ with }r\ge 2$ we get a better bound.
\end{remark}

\subsection{The $4$-Harborth Constant for $C_{2}^{r}$}\hspace{0.6cm} 

We start with a proposition that provides a lower bound for $g^{4}(C_2^r).$

\begin{proposition}\label{t1}
We have $g^{4}(C_2^r)\geq r+\left\lfloor\frac{r}{2} \right\rfloor + \left\lfloor \frac{r-4}{3} \right\rfloor +1.$  
\end{proposition} 

\begin{proof}
Consider the set 
\[\displaystyle{S=T \ \prod_{i=1}^{r}e_i\prod_{i=1}^{\left\lfloor \frac{r}{2}\right\rfloor}(e_{2i-1}+e_{2i})
},\mbox{ with }T=\prod_{i\in B}g_{i},\]
where $g_{i}=\sum_{k=i}^{i+6}e_{k}$ and $B=\left\{1+3j: 0\le j\le \frac{r-7}{3}\right\}.$ Note that $|S|= r+\left\lfloor\frac{r}{2} \right\rfloor + \left\lfloor \frac{r-4}{3} \right\rfloor.$ Clearly $0\not\in \sum_{ 4}(T)\cup \sum_{ 4}(ST^{-1})$. Moreover, any element of $\sum_{\le 3}(T)$ will require at least six elements of canonical basis and hence we are through.




\end{proof}

To prove the next theorem, we need two auxiliary results. The first is a particular case of Theorem 4.1(iii) in \cite{AGS}.

\begin{lemma}\label{adhi}
Let $C_2^r$ be an elementary abelian $2$-group and let $S\subset C_2^r$. If $|S|>r+1,$ then $S$ contains a nontrivial $k$-zero-sum subset such that $k$ is an even number and $2\le k\le r+2$.
\end{lemma}

\begin{lemma}\label{maiort3}
Let $T_1,\dots,T_\ell\subset S \subset C_2^6$ be $6$-zero-sum subsets, $\ell\geq3$, $\sigma(S)\neq 0$ and $|S|=10.$ If  $|(T_i,T_j)|\neq3, i\neq j,$ for some $i,j\in\{1,\cdots,\ell\},$ then there is $T'|S$ with $|T'|=4$ such that $\sigma(T')=0.$  
\end{lemma}
\begin{proof}
Without loss of generality suppose $|(T_1,T_2)|\neq3.$  As $|S|=10,$ it follows that $|(T_i,T_j)|\neq1.$ As $S$ is a subset of $C_2^6,$ it follows that $|(T_i,T_j)|\neq5.$ Hence, $|(T_i,T_j)|\in \{2,3,4\}.$ If $|(T_i,T_j)|=4,$ then we get a $4$-zero-sum subset $T'$ of $S.$ Now, suppose $|(T_1,T_2)|=2.$ If $|(T_3,T_i)|=2$ for all $i\in\{1,2\}$ then $|T_3|\leq4$, a contradiction. If $|(T_3,T_1)|=2$  and $|(T_3,T_2)|=3,$ then $|T_3|\leq5$, a contradiction. As $|(T_1,T_2)|=2,$ it follows that there is $S_1=T_1(T_1,T_2)^{-1}T_2(T_1,T_2)^{-1}$ such that $S_1$ is a $8$-zero-sum subset. Suppose $|(T_3,T_1)|=3$ and $|(T_3,T_2)|=3.$ Then,
\begin{enumerate}
\item if $|(T_1,T_2,T_3)|\neq0,$ then $|T_3|\leq5,$ a contradiction;
\item if $|(T_1,T_2,T_3)|=0,$ then $T_3|S_1,$ $|S_1T_3^{-1}|=2$ and  $\sigma(S_1T_3^{-1})=0,$ a contradiction. 
\end{enumerate}
\end{proof}

\begin{theorem}
We have $g^{4}(C_2^2)=4, g^{4}(C_2^3)=5, g^4(C_2^4)=7, g^{4}(C_2^5)=8$ and $g^{4}(C_2^6)=10$.
\end{theorem}
\begin{proof}
It is trivial to see $g^{4}(C_2^2)=4$. We start proving $g^{4}(C_2^3)=5$. By Proposition \ref{t1}, $g^{4}(C_2^3)\ge 5$.
By Lemma \ref{adhi} one can see a zero-sum subset has to have size 2 or 4. Clearly 2 is impossible, as a consequence we are done. By Proposition \ref{t1} $g^4(C_2^4)\geq 7$. By Lemma \ref{adhi} one can see an $\ell$-zero-sum subset $T,$ where $\ell$ is equal to 2,4 or 6. Clearly $|T|\neq 2$. If $|T|=4$ then done. Suppose $|T|=6.$ Set $a|T$. By Lemma \ref{adhi} $Sa^{-1}$ has an $\ell$-zero-sum subset $T',$ where $\ell$ is equal to 2,4 or 6. Only case which has to be resolved is $|T'|=6$. But, in that case $|(T,T')|=5$ and hence $S(T,T')^{-1}$ is a $2$-zero-sum subset, which is absurd. Hence we are done.

By Proposition \ref{t1}, $g^{4}(C_2^5)\ge 8$. As above one get two $6$-zero-sum subset, say $T$ and $T',$ and $|(T,T')|\in\{4,5\}$. In any other case, we are done.

By Proposition \ref{t1}, $g^{4}(C_2^6)\ge 10$. Let $S\subset C_2^6$ be of size 10. 
By Lemma \ref{adhi} one can see an $\ell$-zero-sum subset $T$ with 
$\ell\in\{2,4,6,8\}$. We may assume $\ell\in\{6,8\}$.
If $\sigma(S)=0$ then we are clearly done.

Assume $\sigma(S)\neq 0$. 
By Lemma \ref{adhi} one can obtain three  $\ell$-zero-sum subset, with $\ell\in\{6,8\},$ say, $T_i, i\in\{1,2,3\}.$ Also one can assume that at least two of them will have size 6, say $T_1$ and $T_2,$ since if we have two set of size 8 we can obtain a $4$-zero-sum subset. Using similar arguments we get $|(T_1,T_2)|\in\{2,3\}.$ 

We divided the analysis into two cases.

\textbf{Case 1.}
Suppose $|T_3|=8$.
Let $a|(T_1,T_2)$ and $b|T_3,$ with $a\neq b.$ By Lemma \ref{adhi} one can obtain a $6$-zero-sum subset of $S(ab)^{-1}$, say $T_4$.


By Lemma \ref{maiort3}, if $|(T_i, T_j)|\neq 3,$ for some $i\neq j$ and $i,j\in\{1,2,4\},$ we are done. Further, we can assume $|(T_i, T_j)|=3$ for $i,j\neq 3$ and $i\neq j.$ 
Hence $|(T_4,T_i)|=3$ for all $i\in\{1,2\}$.

If $|(T_i, T_3)|\neq 4,$ for all $i\in\{1,2,4\},$ we are done. Thus, we may assume $|(T_i, T_3)|= 4$ for $i\neq 3$. Set $T_5=(T_1, T_2)$.
Note that $0\le |(T_5,T_4)|\le 3$. If $|(T_5,T_4)|=2$ then $|T_4|\le 5,$ a contradiction. If $|(T_5,T_4)|=3,$ then $|T_4|\le 4$ , a contradiction, since $|(T_4,T_1)|=|(T_4,T_2)|=3.$ Hence $|(T_5,T_4)|\in\{0,1\}$. 

\textbf{Subcase 1.}
Suppose $|(T_5,T_4)|=0$. 
Set $S=T\cdot g$, where $T=T_1 T_2 T_5^{-1}$.
Since $|(T_i,T_4)|=3$ for $i=1\mbox{ or }2$ we get $T_4|T$. Let $x=|(T_3,T)|$ and $y=|(T_3,T_5)|$. Note that $x=7\mbox{ or }8$ and $y=1,$ since $|(T_i, T_3)|= 4$ for $i\neq 3.$ As $|(T_5,T_4)|=0,$ it follows that $|(T_3,T_4)|= 6$, a contradiction.

\textbf{Subcase 2.}
Suppose $|(T_5,T_4)|=1$. Set $S=T\cdot g$, where $T=T_1 T_2 T_5^{-1}$.

Clearly $g|T_4$. 
Arguing as above, we get $|(T_4,T_3)|\ge 5$, a contradiction.


\textbf{Case 2.}
\textbf{$|T_3|=6.$} 
As in Case 1 we get one more $6$-zero-sum subset, say $T_4.$ By Lemma \ref{maiort3},  $|(T_i,T_j)|=3$ for all $i,j\in\{1,2,3,4\},$ with $i\neq j.$ Set $S=T\cdot g$, where $T=T_1T_2 T_5^{-1}$ with $T_5=(T_1,T_2).$ Note that $0\le |(T_i,T_5)|\le 3, $ for $i\in\{3,4\}$. If $|(T_i,T_5)|=2,3,$ for $i\in\{3,4\},$ then $|T_i|\le 5$, a contradiction. Therefore, $|(T_i,T_5)|\le 1$ for $i\in\{3,4\}$. 

\textbf{Subcase 1.} $|(T_4,T_5)|=0$.
Clearly, $T_4=(T_1T_5^{-1})(T_2T_5^{-1})$.
If $|(T_3,T_5)|=0,$ then $|(T_4,T_3)|=5 \mbox{ or }6$, a contradiction.

\textbf{Subcase 2.} $|(T_4,T_5)|=1$.
If $|(T_3,T_5)|=0$ then $|(T_4,T_3)|\ge4,$ a contradiction.
If $|(T_3,T_5)|=1,$ then as $|(T_i,T_j)|=3$ for all $i\neq j$ we get a $6$-zero-sum subset, namely $T_6 = (T_3(T_3,T_4)^{-1})(T_4(T_3,T_4)^{-1}).$ Note that either $|(T_6,T_1)|=4$ or $|(T_6,T_2)|=4$ or $|(T_6,T_4)|=4,$ a contradiction.


\end{proof}

\subsection{The $k$-Harborth Constant for $C_p$} \hspace{0.6cm}

To obtain a upper bound for $g^k(C_p)$ we used the generalized Erd\H{o}s-Heilbronn conjecture (1964) proved by Dias da Silva and Hamidoune (see theorem below) in 1994.

\begin{theorem}[\cite{sh}]\label{eh}
 If $p$ is a prime and $X$ is a subset of $C_p,$ then $|\hat{k}X|\geq\min\{k|X|-k^2+1,p\}$ for $\hat{k}X=\{x_1+\cdots+x_k|x_i\in X,x_i\neq x_j\}.$
\end{theorem}

%
%

\begin{theorem}\label{ti}
Let $p$ be a prime number such that 
\[p>\left\{
\begin{array}{lll}
m^2&\mbox{ if } & k=2m, m\geq2\\
\frac{2m^3+5m^2-8m-3}{2(m+1)}&\mbox{ if } & k=2m+1, m\geq1
\end{array}
\right.
\]
and $p\equiv r\pmod k.$ 
\begin{enumerate}
\item[(i)]If $k$ is even with $r\in\{1,m+1,\dots,k-1\},$ then
\[
g^k(C_p)=\begin{cases}
\dfrac{p-1}{k}+k, \mbox{ if } r=1\\
\dfrac{p-r}{k}+k+1, \mbox{ if } r\neq1,
\end{cases}
\] 
\item[(ii)] If $k$ is odd with $r\in\{1,\dots,k-1\},$ then
\[
\begin{cases}
g^k(C_p)=\dfrac{p-1}{k}+k, \mbox{ if } r=1\\
\dfrac{p-r}{k}+k\leq g^k(C_p)\leq\dfrac{p-r}{k}+k+1, \mbox{ if } r\neq1.
\end{cases}
\]
\end{enumerate} 
\end{theorem}

\begin{proof}
Using Theorem \ref{eh} one can easily see that any set of size $\left\lceil\dfrac{p+k^2-1}{k}\right\rceil$ has a $k$-zero-sum subset. Therefore, 
 \[
g^k(C_p)\leq\begin{cases}
\dfrac{p-1}{k}+k, \mbox{ if } r=1\\
\dfrac{p-r}{k}+k+1, \mbox{ if } r\neq1.
\end{cases}
\]  
We divide the proof in two cases.\\
{\bf Case 1:} We consider $4\leq k<\dfrac{p-1}{m}$ and the set
\[
S=\begin{cases}
\{0,1,\dots,m-1,p-\left\lceil\frac{p-1}{k}\right\rceil-m+1,\dots,p-1\}, \mbox{ if } k=2m, m\geq2\\
\{1,\dots,m,p-\left\lceil\frac{p-1}{k}\right\rceil-m+1,\dots, p-1\}, \mbox{ if } r\neq1, k=2m+1, m\geq1\\
\{1,\dots,m,p-\left\lceil\frac{p-1}{k}\right\rceil-m,\dots, p-1\}, \mbox{ if } r=1, k=2m+1, m\geq1\\
\end{cases}
\]
If $r\neq1$ set $S'=S\setminus\{0,1,\dots,m-1\}$ if $k$ is even and $S'=S\setminus\{1,\dots,m\}$ if $k$ is odd. The size of $S'$ is $\left\lceil\frac{p-1}{k}\right\rceil+m-1.$ As $k<\dfrac{p-1}{m},$ it follows that $|S'|\geq k.$ 
For $k$ even, suppose $r\in\{m+1,\dots,k-1\}$ and note that any $k$-subsum of $S'$ cannot be multiple of $p,$ since any $k$-subsum of $p-S'=\{1,\dots, \left\lceil\frac{p-1}{k}\right\rceil+m-1\}$ belong to the set $\{\frac{k(k+1)}{2},\dots, p-r+m\}$ and if $r=1,$ any $k$-subsum of $p-S'$ belong to the set $\{\frac{k(k+1)}{2},\dots, p-m\}.$ Now, for $k$ odd, any $k$-subsum of $S'$ cannot be multiple of $p,$ since any $k$-subsum of $p-S'=\{1,\dots, \left\lceil\frac{p-1}{k}\right\rceil+m-1\}$ belong to the set $\{\frac{k(k+1)}{2},\dots, p-r\}.$ 
For $k$ odd and $r=1,$ size of $S'$ is $\left\lceil\frac{p-1}{k}\right\rceil+m.$ In this case, any $k$-subsum of $S'$ cannot be multiple of $p,$ since any $k$-subsum of $p-S'=\{1,\dots, \left\lceil\frac{p-1}{k}\right\rceil+m\}$ belong to the set $\{\frac{k(k+1)}{2},\dots, p-1\}.$ Therefore, if there is a $k$-zero-sum subset of $S,$ then we need to use elements of $S'$ and $SS'^{-1}.$ 

We write $S=S'T,$ where 
\[
S'=\begin{cases}
\{p-\left\lceil\frac{p-1}{k}\right\rceil-m+1,\dots,p-1\}, \mbox{ if } r\neq1\\
\{p-\left\lceil\frac{p-1}{k}\right\rceil-m, \dots,p-1\}, \mbox{ if } r=1
\end{cases}
\]
and 
\[
T=\begin{cases}
\{0,1,\dots,m-1\}, \mbox{ if } k=2m, m\geq2\\
\{1,\dots,m\}, \mbox{ if } k=2m+1, m\geq1.
\end{cases}
\]
Now, we consider $k=\ell+\nu,$ $1\leq\ell\leq m\leq\nu<k$ $(k \mbox{ even}), 1\leq\ell\leq m$ and $m+1\leq\nu<k$ $(k \mbox{ odd}).$ We observe that any $\ell$-subset of $T$ has sum 
\[x\in\left\{\frac{\ell(\ell-1)}{2},\dots,\frac{m(m-1)}{2}-\frac{(m-\ell-1)(m-\ell)}{2}\right\}\, (k \mbox{ even}) \mbox{ or }\]
 \[x\in\left\{\frac{\ell(\ell+1)}{2},\dots,\frac{m(m+1)}{2}-\frac{(m-\ell)(m-\ell+1)}{2}\right\}\, (k \mbox{ odd}).\]
We write $\nu=m+a,$ where $0< a\leq m$ if $k$ is odd and $0\leq a<m$ if $k$ is even. So, for $k$ even or $k$ odd with $r\neq1,$ any $\nu$-subset of $S'$ has sum 
\[y\in\left\{\nu p-\nu\left\lceil\frac{p-1}{k}\right\rceil-\frac{m(m-1)}{2}+\frac{a(a+1)}{2}, \dots, \nu p-\frac{\nu(\nu+1)}{2}\right\} \]
In case $k$ odd with $r=1,$ any $\nu$-subset of $S'$ has sum 
\[y\in\left\{\nu p-\nu\left\lceil\frac{p-1}{k}\right\rceil-\frac{m(m+1)}{2}+\frac{a(a+1)}{2}, \dots, \nu p-\frac{\nu(\nu+1)}{2}\right\}.\]
For $S$ to have a $k$-zero-sum subset, $x+y$ must be divisible by $p,$ i. e., we must have $x-y_1$ divisible by $p,$ where 

$y_1\in\left\{\frac{\nu(\nu+1)}{2}, \dots, \nu\left\lceil\frac{p-1}{k}\right\rceil+\frac{m(m-1)}{2}-\frac{a(a+1)}{2}\right\}\, (k \mbox{ even or } k \mbox{ odd with } r\neq1) $
\[\mbox{ and }y_1\in\left\{\frac{\nu(\nu+1)}{2}, \dots, \nu\left\lceil\frac{p-1}{k}\right\rceil+\frac{m(m+1)}{2}-\frac{a(a+1)}{2}\right\}(k \mbox{ odd with } r=1).\]
Notice that $\nu=k-\ell,$ and

$\begin{array}{ll}
y_1-x\in & \left\{\dfrac{\nu(\nu+1)}{2}-\dfrac{m(m-1)}{2}+\dfrac{(m-\ell-1)(m-\ell)}{2}, \dots, \nu\left\lceil\dfrac{p-1}{k}\right\rceil+\right. \\

 &\left. \ \  +\dfrac{m(m-1)}{2}- \dfrac{a(a+1)}{2}-\dfrac{\ell(\ell-1)}{2}\right\}(k \mbox{ even}) \mbox{ or }\\
\end{array}$

$\begin{array}{ll}
y_1-x\in & \left\{\dfrac{\nu(\nu+1)}{2}-\dfrac{m(m-1)}{2}+\dfrac{(m-\ell)(m-\ell+1)}{2}, \dots, \nu\left\lceil\dfrac{p-1}{k}\right\rceil+\right.\\

& \ \ \left. +\dfrac{m(m-1)}{2}-\dfrac{a(a+1)}{2}-\dfrac{\ell(\ell+1)}{2}\right\}\, (k \mbox{ odd with }r\neq1) \mbox{ or }

\end{array}$

$\begin{array}{ll}
y_1-x\in & \left\{\dfrac{\nu(\nu+1)}{2}-\dfrac{m(m+1)}{2}+\dfrac{(m-\ell)(m-\ell+1)}{2}, \dots, \nu\left\lceil\dfrac{p-1}{k}\right\rceil\right.\\
& \ \ \left. +\dfrac{m(m+1)}{2}-\dfrac{a(a+1)}{2}-\dfrac{\ell(\ell+1)}{2}\right\}(k \mbox{ odd with }r=1).
\end{array}$

We observe that 
\[
\nu\geq m \mbox{ and } 1\leq\ell\leq m \Rightarrow \frac{\nu(\nu+1)}{2}-\frac{m(m-1)}{2}+\frac{(m-\ell-1)(m-\ell)}{2}>0\,\] ($k$  even)  and

\[
\nu> m \mbox{ and } 1\leq\ell\leq m \Rightarrow \frac{\nu(\nu+1)}{2}-\frac{m(m-1)}{2}+\frac{(m-\ell)(m-\ell+1)}{2}>0\,\] ($k$  odd with $r\neq1$) and

\[
\nu>m \mbox{ and } 1\leq\ell\leq m+1 \Rightarrow \frac{\nu(\nu+1)}{2}-\frac{m(m+1)}{2}+\frac{(m-\ell)(m-\ell+1)}{2}>0\,\] ($k$  odd with $r=1$).

Let $k$ be even. Notice that $\ell=m-a,$ with $0\leq a<m,$ and the smallest value of 
\[
\frac{a(a+1)}{2}+\frac{\ell(\ell-1)}{2}=\frac{a(a+1)}{2}+\frac{(m-a)(m-a-1)}{2}
\]  
 it happens when $a=\frac{m-1}{2}.$ Therefore, the largest value of the expression 
\[\nu\left\lceil\frac{p-1}{k}\right\rceil+\frac{m(m-1)}{2}-\frac{a(a+1)}{2}-\frac{\ell(\ell-1)}{2}\]
it happens when $a=\frac{m-1}{2}.$ So, if $r=1,$
\[
\begin{array}{lll}
\nu\left\lceil\frac{p-1}{k}\right\rceil+\frac{m(m-1)}{2}-\frac{a(a+1)}{2}-\frac{\ell(\ell-1)}{2} & \leq & \frac{3m-1}{4m}(p-1)+\frac{m^2-m}{2}-\frac{m^2-1}{4}\\
&&\\
& = & \frac{(3m-1)p}{4m}+\frac{m^2-m}{2}-\frac{m^2-1}{4}-\frac{3m-1}{4m}.
\end{array}
\]
For the last value to be less than $p,$ we must have $p>\frac{m^3-2m^2-2m+1}{m+1}.$ But this is guarantee by hypothesis, since $p>m^2.$
Therefore, $p\nmid x-y_1.$\\
If $r\in\{m+1,\dots,k-1\},$ then
\[
\begin{array}{lll}
\nu\left\lceil\frac{p-1}{k}\right\rceil+\frac{m(m-1)}{2}-\frac{a(a+1)}{2}-\frac{\ell(\ell-1)}{2} & \leq & \frac{3m-1}{2} (\frac{p-r}{k})+\frac{3m-1}{2}+\frac{m^2-m}{2}-\frac{m^2-1}{4}\\
&&\\
& < & \frac{(3m-1)p}{4m}-\frac{3m-1}{4}+\frac{m^2+2m-1}{2}-\frac{m^2-1}{4},
\end{array}
\] 
For the last value to be less than $p$ we must have $p>m^2.$ But this is guarantee by hypothesis.
Therefore, $p\nmid x-y_1.$\\

Let $k$ be odd. Notice that $\ell=m-a,$ with $0\leq a< m,$ and the smallest value of 
\[
\frac{a(a+1)}{2}+\frac{\ell(\ell+1)}{2}=\frac{a(a+1)}{2}+\frac{(m-a)(m-a+1)}{2}
\]  
 it happens when $a=\frac{m+1}{2}.$ For $r=1$ the largest value of the expression 
\[\nu\left\lceil\frac{p-1}{k}\right\rceil+\frac{m(m+1)}{2}-\frac{a(a+1)}{2}-\frac{\ell(\ell+1)}{2}\]
it happens when $a=\frac{m+1}{2}.$ So
\[
\begin{array}{lll}
\nu\left\lceil\frac{p-1}{k}\right\rceil+\frac{m(m+1)}{2}-\frac{a(a+1)}{2}-\frac{\ell(\ell+1)}{2} & \leq &\frac{3m+1}{4m+2}(p-1)+\frac{m^2+m}{2}-\frac{\frac{2m^2+4m+2}{4}}{2}\\
&&\\
& = & \frac{(3m+1)p}{4m+2}+\frac{m^2-1}{4}-\frac{3m+1}{4m+2}.
\end{array}
\]
For the last value to be less than $p,$ we must have $p>\frac{2m^3+m^2-8m-3}{2(m+1)}.$ But this is guarantee by hypothesis, since $p>\frac{2m^3+5m^2-8m-3}{2(m+1)}.$
Therefore, $p\nmid x-y_1.$\\
If $r\in\{2,\dots,k-1\},$ then
\[
\begin{array}{lll}
\nu\left\lceil\frac{p-1}{k}\right\rceil+\frac{m(m-1)}{2}-\frac{a(a+1)}{2}-\frac{\ell(\ell+1)}{2} & \leq & \frac{(3m+1)(p-r)}{4m+2}+\frac{m^2+2m+1}{2}-\frac{\frac{2m^2+4m+2}{4}}{2}\\
&&\\
&\leq & \frac{(3m+1)}{4m+2}p-\frac{3m+1}{2m+1}+\frac{m^2+2m+1}{4}.
\end{array}
\]
For the last value to be less than $p,$ we must have $p>\frac{2m^3+5m^2-8m-3}{2(m+1)}.$ But this is guarantee by hypothesis.
Therefore, $p\nmid x-y_1.$\\
{\bf Case 2:} We consider $\dfrac{p-1}{m}<k<p$ and the same sets of Case 1. \\

In this case $|S'|<k.$ If there is a $k$-zero-sum subset of $S,$ then we need to use elements of $S'$ and $T.$ Here the proof is the same as in Case 1.


\end{proof}

\begin{remark}
The Theorem \ref{ti} provides $g^3(C_p)=\left\lceil\frac{p+8}{3} \right\rceil$ for $p\equiv1 \pmod3,$  $g^4(C_p)=\left\lceil\frac{p+15}{4} \right\rceil$ for all prime $p,$ $g^5(C_p)=\left\lceil\frac{p+24}{5} \right\rceil$ for $p\equiv1 \pmod5,$  
$g^6(C_p)=\left\lceil\frac{p+35}{6} \right\rceil$ for all prime $p>9,$ etc.
\end{remark}

\begin{remark}
We have $g^6(C_7)=7,$ since the sum of set $S=\{0,1,2,4,5,6\}$ is non zero. In fact, the proof of Theorem \ref{ti} provides $g^6(C_p)=\left\lceil\frac{p+35}{6} \right\rceil$ for all $p\equiv1 \pmod 6$ and $g^8(C_p)=\left\lceil\frac{p+63}{8} \right\rceil$ for all $p\equiv1 \pmod 8.$ The Theorem \ref{ti} provides $g^8(C_p)=\left\lceil\frac{p+63}{8} \right\rceil$ for all $p>16.$ But, for $p=11,$ the set $S=\{0,1,2,4,6,7,8,9,10\}$ or the set $S=\{0,1,2,5,6,7,8,9,10\}$ is an $8$-zero-sum free set. Therefore, $g^8(C_{11})=10.$ For $p=13,$ the set $S=\{0,1,2,3,8,9,10,11,12\}$ is an $8$-zero-sum free set. Therefore, $g^8(C_{13})=10.$ 
\end{remark}

\subsection{The $k$-Harborth Constant, with $k\in\{3,4\},$ for group $C_n$}\hspace{0.6cm}

We begin presented an upper bound for $g^3(C_n),$ where $n>2$ is an even number.

\begin{proposition}
We have $g^3(C_n)\le \frac{n}{2}+3,$ for any even number $n\ge4.$
\end{proposition}
\begin{proof}
Let $S\subset C_n$ with $|S|=\frac{n}{2}+3.$ Set $a\in S$. Consider a set
\[A = \{a+x:x\in S\setminus\{a\}\}.\]
Since $|C_n\setminus -S|=\frac{n}{2}-3$ and $|A|=\frac{n}{2}+2$, we get $|A\cap -S|\ge 5.$ Hence, $a+x_i=-y_i$, where $x_i\in S\setminus\{a\} \mbox{ and }y_i\in S,\mbox{ for }i\in\{1,2,3,4,5\}$. If 
$|\{a,x_i,y_i\}|=3$ for some $i\in\{1,2,3,4,5\}$ then we have a $3$-zero-sum subset of $S.$. Hence we are done. Assume on the contrary that, 
$|\{a,x_i,y_i\}|<3$ for all $i\in\{1,2,3,4,5\}.$ Without loss of generality we may assume that $y_i\neq a,\forall i\le 4$. Clearly, $x_i = y_i$ 
for all $i\in\{1,2,3,4\}.$ Hence $2(x_i-x_j)=0$ for all $i,j\in\{1,2,3,4\}, i\neq j.$ As a consequence, $x_i-x_j=\frac{n}{2}$ for all $i,j\in\{1,2,3,4\}, i\neq j.$ Hence $x_1=x_2$, a contradiction. Hence we conclude that we have a $3$-zero-sum subset of $S.$ Hence we are done.

\end{proof}

Now we present equalities for particular values of $n$.
\begin{proposition}
We have 
\begin{enumerate}
\item[(i)] $g^3(C_4)=4$ and $g^3(C_7)=g^3(C_6)=g^3(C_5)=5;$
\item[(ii)] $g^3(C_8)=6$ and $g^3(C_9)=7;$ 
\item[(iii)] $g^3(C_{12})=g^3(C_{11})=g^3(C_{10})=7.$ 
\end{enumerate}
\end{proposition}
\begin{proof}
For all cases we will use the following idea: we write $C_n=B \cup  A \cup -A,$ where $B=\{x\in C_n\ | \ 2x=0\}.$ \\ 
\begin{enumerate}
\item It is clear to see that $g^3(C_4)=4$ and $g^3(C_6)=g^3(C_5)=5.$ Now, $g^3(C_7)=5,$ by Theorem \ref{ti}.
\item The subsets $S_1=\{1,2,4,6,7\}\subset C_8$ and $S_2=\{1,2,4,5,7,8\}\subset C_9$ are $3$-zero-sum free sets. Hence $g^3(C_8)\ge6$ and $g^3(C_9)\ge7.$ Now, we consider any subset $S\subset C_8 (C_9)$ of size $6$ $(7).$ In both cases, if $0\in S,$ we get a $3$-zero sum subset in $C_8 (C_9),$ since $|A|=|-A|=3$ for $C_8$ and $|A|=|-A|=4$ for $C_9$ implies that there is $x\in S$ such that $-x\in S.$ If $0\not\in S$ we consider the representation of 8 (9) and 16 (18) as sum of three distinct elements of $C_8 (C_9).$ 
We observe that $\{1,2,5\}\cap\{3,6,7\}=\emptyset$ ($\{1,2,6\}\cap\{3,7,8\}=\emptyset).$ By Pigeon hole principle one can easily see that $C=\{1,2,5\} (C = \{1,2,6\})\subset S$ or $-C (-C)\subset S$ for any subset $S\subset C_8 (C_9)$ of size $6$ $(7),$ with $0\not\in S.$ Therefore, we get a $3$-zero-sum subset of $S\subset C_8 (C_9).$     
\item Note that $S=\{2,3,4,6,7,8\}\subset C_{10}$ is a $3$-zero-sum free set. Hence $g^3(C_{10})\ge 7.$ Let $S$ be subset of $C_{10}$ of size $7.$ If $0\in S$ then it is clear that $a,-a\in S\setminus \{0,5\}$. Hence $\{0,a,-a\}$ is a $3$-zero-sum subset of our interest. Assume that $0\not\in S$. If $5\not\in S$ then the fact that $\{1,2,7\}\cap \{9,8,3\}=\emptyset$ implies the existence of a $3$-zero-sum subset. Next we assume that $5\in S$. Observe that $\{1,4\},\{3,2\} \mbox{ and }\{7,8\}$ are pairwise disjoint one can conclude that there is a $3$-zero-sum subset.

For $C_{11},$ notice that $S=\{3,4,5,6,7,8\}$ is a $3$-zero-sum free set. Hence $g^3(C_{11})\ge 7.$ Now, by Theorem \ref{ti} $g^3(C_{11})\le 7.$
%

For $C_{12},$ notice that $S=\{1,2,5,7,10,11\}$ is a $3$-zero-sum free set. Hence $g^3(C_{12})\ge 7.$ Let $0\in S\subset C_{12}$ be any set of 7 elements. 

As $0\in S$ then it is clear that $\{a,-a\}\subset S\setminus \{0,6\}$ or $\{\epsilon_11,\epsilon_22,\dots,\epsilon_55\}\subset S,$ where $\epsilon_i\in\{-1,1\}$ for $i\in\{1,2,\dots,5\}.$ If $\{a,-a\}\subset S,$ we get $\{0,a,-a\}$ to be a $3$-zero-sum subset. 

On the contrary, we will divide the proof in two cases. Clearly, in this case $6\in S$. If $\{\epsilon_11,\epsilon_15\}$ or $\{\epsilon_22,\epsilon_24\}$ are contained in $S$ then we are done. So we may assume on the contrary.\\
\textbf{Case 1.} Suppose $\{1, -5\}\subset S$. If $4\in S,$ then $\{1,4,-5\}$ is a $3$-zero-sum subset of $S$. Otherwise $\{2,-4\}\subset S.$ If $-3\in S,$ then $\{1,2,-3\}$ is a $3$-zero-sum subset of $S$. Otherwise $3\in S$ and $\{1,3,-4\}$ is a $3$-zero-sum subset of $S$.\\
\textbf{Case 2.} Suppose $\{-1, 5\}\subset S$. As we argued in Case 1, one can easily see that $\{-1,-4,5\}$
or $\{-1,-2,3\}$ or $\{-1,-3,4\}$ is a $3$-zero-sum subset.\\


If $a\in S,$ with $a\in\{4,-4\},$ then consider a set $T=S-a$. Notice that, $0\in T$. As we have seen above there is a $3$-zero-sum subset of $T,$ since $3a=0.$ Assume from now on $S\cap \{0,4,-4\}=\emptyset.$ Suppose $6\in S$. If $\{1,5\}\subset S$ then
$\{6,1,5\}$ is a $3$-zero-sum subset. If 
$-\{1,5\}\subset S,$ then
$\{6,-1,-5\}$ is a $3$-zero-sum subset. Hence $S$ must be one among the $\{6,1,2,3,-1,-2,-3\}$ or 
$\{6,1,2,3,-5,-2,-3\}$ or $\{6,5,2,3,-1,-2,-3\}$
 or $\{6,5,2,3,-5,-2,-3\}$. In any case it is clear to see a $3$-zero-sum subset. 
Suppose $6\not\in S$. Clearly $\{\epsilon_11,\epsilon_22,\epsilon_33,\epsilon_55\}\subset S,$  where $\epsilon_i\in\{-1,1\}$ for $i\in\{1,2,3,5\}.$ Now, $S=\{\pm1,\pm2,\pm3,\pm5\}\backslash\{x\},$ where $x=-\epsilon_ii$ for some $i\in\{1,2,3,5\}.$ Therefore, there is a $3$-zero-sum subset of $S.$

\end{enumerate}
\end{proof}

\begin{proposition}
We have $g^3(C_{18})=10.$
\end{proposition}
\begin{proof}
According to Proposition \ref{lower}, $g^3(C_{18})\geq10.$ Now, we consider $S$ a subset of $C_{18}$ of size 10. If $0\in S$ then it is clear that $\{a,-a\}\subset S\setminus \{0,9\}$ or $\{\epsilon_11,\epsilon_22,\dots,\epsilon_88\}\subset S,$ where $\epsilon_i\in\{-1,1\}$ for $i\in\{1,2,\dots,8\}.$ If $\{a,-a\}\subset S,$ we get $\{0,a,-a\}$ to be a $3$-zero-sum subset. 

On the contrary, we will divide the proof in two cases. Clearly, in this case $9\in S$. If $\{\epsilon_11,\epsilon_18\}$ or $\{\epsilon_22,\epsilon_27\}$ or $\{\epsilon_33,\epsilon_36\}$ or $\{\epsilon_44,\epsilon_45\}$ are contained in $S,$ then we are done. So we may assume on the contrary.\\
\textbf{Case 1.} Suppose $\{1, -8\}\subset S$. If $7 \in  S$ then $\{1, -8, 7\}$ is a $3$-zero-sum subset. Otherwise $\{2, -7\}\subset S$. If $5\in S$ we are done. Otherwise $\{4, -5\}\subset S$ and $\{1, 4, -5\}$ is a $3$-zero-sum subset.\\
\textbf{Case 2.} Suppose $\{-1,8\}\subset S.$ As we argued in Case 1, one can easily see that $\{-1,8,-7\}$
or $\{-5,-2,7\}$ or $\{5,-4,-1\}$ is a $3$-zero-sum subset.

If $a\in S,$ with $a\in\{6,-6\},$ then consider the set $T=S-a$. Notice that, $0\in T$. As we have seen above there is a $3$-zero-sum subset of $T,$ since $3a=0.$ 
Assume from now on $S\cap \{0,6,-6\}=\emptyset.$ If $\{1,2,3,4,5,7,8\}\subset S$ or  $-\{1,2,3,4,5,7,8\}\subset S,$ then $\{3,7,8\}$ or $\{-3,-7,-8\}$ is a $3$-zero-sum subset of $S.$ Suppose $9\in S$. As  $\pm1\pm8=\pm9,$ $\pm2\pm7=\pm9$ and $\pm4\pm5=\pm9$ and $|S|=10,$ it follows that we get a zero-sum subset of our interest. Suppose $9\not\in S$. If $\{1,2,3,4\}\subset S,$ then $S\cap \{-3,-4,-5,-7\}=\emptyset.$ As $|S\cap\{3,7,8\}|<3,$ we have a contradiction since $|S|=10.$ If $\{1,3,4\}\subset S$ and $2\not\in S,$ then $S\cap \{-4,-5,-7\}=\emptyset$ and $|S\cap\{3,7,8\}|<3,$ a contradiction since $|S|=10.$ If $\{2,3,4\}\subset S$ and $1\not\in S,$ then $S\cap \{-5,-7\}=\emptyset,$ $|S\cap\{3,7,8\}|<3,$ and $|S\cap\{-1,-2,-3\}|<3,$ a contradiction since $|S|=10.$ If $\{1,2,4\}\subset S$ and $3\not\in S,$ then $S\cap \{-3,-5\}=\emptyset,$ $|S\cap\{-1,-7,8\}|<3$ and $|S\cap\{1,7,-8\}|<3$ a contradiction since $|S|=10.$ If $\{1,2,3\}\subset S$ and $4\not\in S,$ then $S\cap \{-3,-4\}=\emptyset,$ $|S\cap\{-1,-7,8\}|<3$ and $|S\cap\{1,7,-8\}|<3$ a contradiction since $|S|=10.$
Suppose $\{i,j\}\subset\{1,2,3,4\},$ with $i,j\not\in S.$ Thus, there is $U\subset-\{1,2,3,4,5,7,8\}$ such that $|U|\geq 5$ and $U\subset S.$ As $\{-3,-7,-8\}$ is a $3$-zero-sum subset, it follows that $|U\cap\{-3,-7,-8\}|<3,$ i. e., $|U|=5 \mbox{ or }6.$ Now, if $|S\cap\{5,7\}|=2,$ then $|U\cap \{-1,-2,-4,-5\}|<3$ a contradiction since $|S|=10.$ If $|S\cap \{5,7\}|=1,$ then $|U\cap \{-1,-2,-4,-5\}|<4$ a contradiction since $|S|=10.$ If $|S\cap \{5,7\}|=0,$ we have a contradiction since $|S|=10.$   

\end{proof}

For $k=4,$ we obtain the bounds below for the group $C_n.$

\begin{proposition} \label{l2}
We have $\left\lceil\frac{n+15}{4} \right\rceil \le g^4(C_n)\le\left\lfloor \frac{n+6}{2} \right\rfloor,$ where $n$ is an odd integer with $n\ge 5.$ 

\end{proposition}
\begin{proof}
We consider the set
\[
S=
\begin{cases}
\{0,1,\dots,\frac{n+3}{4},n-2\}, \mbox{ if } n\equiv1\pmod4 \\
\{0,1,\dots,\frac{n+5}{4},n-2\}, \mbox{ if } n\equiv3\pmod4.
\end{cases}
\]

Set $S'=S\setminus\{n-2\}$. Suppose $n\equiv 1\pmod 4$. Note that any subset of $S'$ of size 4 has sum $t$ with $1\le t\le n-3$. Also, note that any subset of $S'$ of size 3 has sum $t$ with $3\le t\le \frac{3(n+3)}{4}-3$.
Hence any subset of $S$ of size 4 has sum $t$ with $n+1\le t\le 2n-1$ or $1\le t\le n-3$.

Suppose $n\equiv 3\pmod 4$. Note that any subset of $S'$ of size 4 has sum $t$ with $1\le t\le n-1$. Also, note that any subset of $S'$ of size 3 has sum $t$ with $3\le t\le \frac{3(n+5)}{4}-3$.
Hence any subset of $S$ of size 4 has sum $t$ with $n+1\le t\le 2n-1$ or $1\le t\le n-1$.
Hence $g^4(C_n)\ge |S|+1$. Observe that, $|S|+1=\left\lceil\frac{n+15}{4} \right\rceil$. Hence we are done.

Next, we shall prove $\left\lfloor \frac{n+6}{2} \right\rfloor \ge g^4(C_n)$. Write $C_n=\{0\}\cup A\cup -A$. By Pigeon hole principle, given a set $S$ of size $\left\lfloor \frac{n+6}{2} \right\rfloor$, $\{a,b,-a,-b\}\subset S$. Hence we are done.
\end{proof}

\begin{proposition} \label{l3}
We have $\left\lceil\frac{n+14}{4} \right\rceil \le g^4(C_n)\le \frac{n}{2}+3,$ where $n$  is an even integer with $n> 5.$  
\end{proposition}

\begin{proof}
We consider the set
\[
S=
\begin{cases}
\{0,1,\dots,\frac{n+4}{4},n-2\}, \mbox{ if } n\equiv0\pmod4 \\
\{0,1,\dots,\frac{n+2}{4},n-2\}, \mbox{ if } n\equiv2\pmod4.
\end{cases}
\]

Set $S'=S\setminus\{n-2\}$. Suppose $n\equiv 0\pmod 4$. Note that any subset of $S'$ of size 4 has sum $t$ with $1\le t\le n-2$. Also, note that any subset of $S'$ of size 3 has sum $t$ with $3\le t\le \frac{3(n+4)}{4}-3$.
Hence any subset of $S$ of size 4 has sum $t$ with $n+1\le t\le 2n-1$ or $1\le t\le n-2$.

Suppose $n\equiv 2\pmod 4$. Note that any subset of $S'$ of size 4 has sum $t$ with $1\le t\le n-4$. Also, note that any subset of $S'$ of size 3 has sum $t$ with $3\le t\le \frac{3(n+2)}{4}-3$.
Hence any subset of $S$ of size 4 has sum $t$ with $n+1\le t\le 2n-1$ or $1\le t\le n-4$.
Hence $g^4(C_n)\ge |S|+1$. Observe that, $|S|+1=\left\lceil\frac{n+14}{4} \right\rceil$. Hence we are done.

Next, we shall prove $\frac{n}{2}+3 \ge g^4(C_n)$. Write $C_n=\{0,\frac{n}{2}\}\cup A\cup -A$. By Pigeon hole principle, given a set $S$ of size $\frac{n}{2}+3$, $\{a,b,-a,-b\}\subset S$. Hence we are done.
\end{proof}

\begin{proposition}\label{l5}
We have $g^k(C_{2n})\leq 2g^k(C_n)-1,$ where $k$ is an even number and $n\geq k$ is an odd number.
\end{proposition}
\begin{proof}
Let $S\subset C_{2n}$ with $|S|=2g^k(C_n)-1$. Write \[C_{2n}=C_2\oplus C_n.\]
By Pigeon hole principle, there exists a subset $S_1\subset S$ where $S_1=\{(a,g_i):1\le i\le l\}$ with $l\ge  g^k(C_n)$.
Clearly, $g_i\neq g_j$ for $i\neq j$. Hence there exists a $k$-zero-sum subset $T|\prod_1^l g_i.$ Since $k$ is even $\{(a,g_i):g_i|T\}$ is a $k$-zero-sum subset of $S$.
\end{proof}
\begin{corollary}\label{co1}
We have $\dfrac{p+7}{2}\leq g^4(C_{2p})\leq 2\left\lceil\dfrac{p+15}{4}\right\rceil-1,$ where $p>3$ is an odd prime number.
\end{corollary}
\begin{proof}
The lower bound is given by Proposition \ref{l3} and upper bound is given by Theorem \ref{ti} and Proposition \ref{l5}.
\end{proof}
\begin{rem}
For $p=5$ or $7,$ the Proposition \ref{l3} provides $g^4(C_{10})\leq 8$ and $g^4(C_{14})\leq 10.$ These upper bounds are better than upper bounds provides by Corollary \ref{co1}. 
\end{rem}

\end{document}